\documentclass[leqno,11pt, a4]{amsart}
\usepackage[usenames,dvipsnames,svgnames,table]{xcolor}

\usepackage[active]{srcltx}
\usepackage{pb-diagram}

\usepackage{dsfont}

\usepackage{mathrsfs}
\usepackage{amsmath}
\usepackage{amssymb}
\usepackage{amscd}
\usepackage{amsthm}
\usepackage[latin1]{inputenc}
\usepackage{graphics}
\usepackage{varioref}

\usepackage[latin1]{inputenc}
\usepackage{graphics}
\usepackage{pdfsync}

\labelformat{enumi}{(#1)}

\newtheorem{theorem}[equation]{Theorem}

\newtheorem{proposition}[equation]{Proposition}
\newtheorem{cor}[equation]{Corollary}
\newtheorem{lemma}[equation]{Lemma}

\newtheoremstyle{named}{}{}{\itshape}{}{\bfseries}{.}{.5em}{\thmnote{#3}#1}
\theoremstyle{named}

\newcommand{\Keler}             {K\"{a}hler }

\newcommand{\liu}{\mathfrak{u}}
\newcommand{\liek}{\mathfrak{k}}

\newcommand{\lieg}{\mathfrak{g}}

\newcommand{\liep}{\mathfrak{p}}

\newcommand{\desudt}[1] []      {\dfrac {\mathrm {d} #1 }{\mathrm {dt}}}
\newcommand{\desudtzero}        {\desudt \bigg \vert _{t=0} }

\newcommand{\sx}{\langle}
\newcommand{\xs}{\rangle}
\newcommand{\scalo}{\sx \cdot , \cdot \xs}

\newcommand{\U} {\operatorname{U}}

\newcommand{\lra}{\longrightarrow}
\newcommand{\C}{\mathbb{C}}
\newcommand{\R}{\mathds{R}}

\renewcommand{\phi}{\varphi}

\newcommand{\meo}{\end{document}}
\newcommand{\mup}{\mu_\liep}

\newcommand{\mupb}{\mu_\liep^\beta}

\begin{document}
\title{A splitting result for real submanifolds of a K\"ahler manifold}
\author{Leonardo Biliotti}
\address{(Leonardo Biliotti) Dipartimento di Scienze Matematiche, Fisiche e Informatiche \\
          Universit\`a di Parma (Italy)}
\email{leonardo.biliotti@unipr.it}
\begin{abstract}
Let $(Z,\omega)$ be a connected K\"ahler manifold with an holomorphic action of the complex reductive Lie group $U^\C$, where $U$ is a compact connected Lie group acting in a hamiltonian fashion. Let $G$ be a closed compatible Lie group of $U^\C$ and let $M$ be a $G$-invariant connected submanifold of $Z$. Let $x\in M$. If $G$ is a real form of $U^\C$, we investigate conditions such that $G\cdot x$ compact implies $U^\C\cdot x$ is compact as well. The vice-versa is also investigated. We also characterize $G$-invariant real submanifolds such that the norm square of the gradient map is constant. As an application, we prove a splitting result for real connected submanifolds of $(Z,\omega)$  generalizing a result proved in \cite{pg}, see also \cite{bg,bs}.
\end{abstract}
 \keywords{gradient map; real reductive Lie groups}

%
 \subjclass[2010]{22E45,53D20}
\thanks{The author was partially supported by Project PRIN  2017 ``Real and Complex Manifolds: Topology, Geometry and holomorphic dynamics'' and by GNSAGA INdAM.}
\maketitle
\section{Introduction}
Let $(Z, \omega)$ be a \Keler manifold. Assume that $U^\C$ acts
holomorphically on $Z$, that $U$ preserves $\omega$ and that there is a
momentum map for the $U$ action on $Z$. This means there is a map $\mu:Z \lra \liu^*$, where $\liu$ is the Lie algebra of $U$ and $\liu^*$ is its dual, which is $U$ equivariant with respect to the given action of $U$ on $Z$ and the coadjoint action $\mathrm{Ad}^*$ of $U$ on $\liu^*$ and satisfying the following condition. Let $\xi \in \liu$. We denote by $\xi_Z$ the induced vector field on $Z$, i.e., $\xi_Z (p)=\desudt \vert_{t=0} \exp (t \xi ) p $. Let $\mu^\xi$ be the function $\mu^\xi (z) := \mu(z) (\xi)$, i.e., the contraction of the moment map along $\xi$. Then
$\mathrm{d} \mu^\xi =\it{i}_{\xi_Z} \omega$.

Let $G$ be  a closed connected subgroup
of $U^\C$ compatible with respect
to the Cartan decomposition of $U^\C${, i.e. $G=K\exp (\liep)$, for $K=U\cap
G$ and $\liep=\lieg \cap i\liu$ \cite{helgason,knapp-beyond}.  The
inclusion $i \liep \hookrightarrow \liu$ induces by restriction a
$K$-equivariant map $\mu_{i \liep}:Z \lra (i \liep)^*$ \cite{heinzner-schwarz-stoetzel,heinzner-stoetzel-global})}.

Let $\scalo$ be a $U$-invariant scalar product on $\liu$. Let $\scalo$ denote also the inner product on $i\liu$ such that $i$ be an isometry  of $\liu$ into $i\liu$.
Hence we may identify $\liu^*$ and $\liu$  by means of $\scalo$ and so we view $\mu$ as a map
$\mu:Z \lra \liu$. Therefore, we may view $\mu_{i\liep}$ as a map $\mup:Z \lra \liep$ as follows:
\[
\sx \mup(x), \beta \xs=-\sx \mu(x),i\beta\xs.
\]
We call $\mup$  the \emph{$G$-gradient map associated with $\mu$}.  We also set $\mupb:= \sx \mup, \beta \xs$. By definition, it follows that $\mathrm{grad}\mupb =\beta_Z$. If $M$ is a $G$-stable locally closed real submanifold of $Z$, we may consider $\mup$ as a mapping $\mup:M \lra \liep$ such that $\mathrm{grad}\mup =\beta_M$, where the gradient is computed with respect to the induced Riemannian metric on $M$. Since $M$ is $G$-stable it follows $\beta_Z(p)=\beta_M(p)$ for any $p\in M$.

Assume that $G$ is a real form of $U^\C$. If $U^\C \cdot x$ is compact, then it is well-known that $G$ has a closed orbit contained in $U^\C \cdot x$ \cite{heinzner-schwarz-stoetzel}. On the other hand, if $G\cdot x$ is closed then it is not in general true that $U^\C \cdot x$ is closed as well \cite{gjt}. In Section \ref{se1}, we investigate conditions such that $G \cdot x$ compact implies $U^\C \cdot x$ is compact. If $G\cdot x$ is compact then we give a necessary condition to $U^\C \cdot x$ be compact. If $M$ is Lagrangian, then $U^\C \cdot x$ being compact implies $G\cdot x$ is a Lagrangian submanifold of $U^\C \cdot x$. Finally, we study the case when $Z$ is $U^\C$-semistable, $M$ is $G$-semistable and is contained in the zero level set of the gradient map of $K^\C$. As an application we proof a well-known result of Birkes \cite{birkes}.

A strategy for analyzing the $G$ action on $M$ is to view the function $\nu_\liep : M \lra \R$, \[\nu_\liep (x)=\parallel \mup (x) \parallel^2\] as a Morse like function. The function $\nu_\liep$ is called the norm square of the gradient map. If $M$ is compact or $\mup$ is proper, then associated to the critical points of $\nu_\liep$ we have $G$-stable submanifold of $M$ that they are strata of a Morse type stratification of $M$ \cite{heinzner-schwarz-stoetzel,kirwan}. In Section \ref{se2}, we investigate under which condition $\nu_\liep$ is constant. The following result has some interest itself.
\begin{proposition}\label{main2}
Let $M$ be a $G$-stable connected submanifold of $Z$ and let $\mup:M \lra \liep$ be the restricted gradient map. Then the square of the gradient map $\nu_\liep:M \lra \R$ is constant  if and only if any $G$ orbit is compact.
\end{proposition}
 By the stratification Theorem \cite{heinzner-schwarz-stoetzel}, it follows that $M$ coincides with a maximal pre-stratum and $\mup(M)=K\cdot \beta$. Moreover, $M=K\times_{K^\beta} \mup^{-1}(\beta)$, where
$K^\beta=\{k\in K:\, \mathrm{Ad}(k)(\beta)=\beta\}$.  Let $x\in \mup^{-1}(\beta)$. By the $K$-equivariance of $\mup$, it follows that the stabilizer $K_x \subseteq K^\beta$. Although $G\cdot x$ is closed, it is not true in general $K_x = K^\beta$. Indeed, let $U$ be a connected, compact semisimple Lie group and let $\rho:U \lra \mathrm{SL}(W)$ be a complex representation. Let $G$ be a noncompact connected semisimple  real form of $U^\C$. It is well known that $U^\C$ has a closed orbit in $\mathbb P(W)$, which is a complex $U$-orbit \cite{guillemin}. Let $\mathcal O$ denote a closed orbit of $U^\C$. If $x\in \mathcal O$ realizes the maximum of the norm squared of the $G$-gradient map  restricted to $\mathcal O$, then $G\cdot x$ is closed  and it is a $K$ orbit \cite{heinzner-schwarz-stoetzel}. Now, $K_x=K\cap U^{\mu(x)}$ and $U^{\mu(x)}=U_x$ since $U\cdot x$ is complex \cite{guillemin}. However, $\mu(x) \notin \liep$ and so $K_x$ does not coincide in general with $K^{\mup(x)}$.

If $M$ is a $U$-invariant compact connected complex submanifold of $(Z,\omega)$, then $\nu_{i\mathfrak u}$ constant is equivalent to $U$ is semisimple and $M=U/U_\beta \times \mu^{-1}(\beta)$. The above splitting is Riemannian \cite{pg} (see
also \cite{bg,bs} for the  same result  under the assumption that $M$ is symplectic). In this paper we prove this splitting result without any assumption on $M$.
\begin{theorem}\label{main2}
Let $M$ be a $U^\C$-stable connected submanifold of $Z$ and let $\mu:M \lra \liu$ be the restricted momentum map. Then the square of the momentum map $\parallel \mu \parallel^2$ is constant  if and only if $U$ is semisimple and $M$ is $U$-equivariantly isometric to the product of a flag manifold  and an embedded, closed submanifold which is acted on trivially by $U$.
\end{theorem}
Assume that $G$ is a real form of $U$. The momentum map of $U$ on $Z$ induces a gradient map $\mu_{i\mathfrak k}$ of $K^\C$ in $Z$. We say that $M$ is $G$-semistable if $M=\{p\in M:\, \overline{G \cdot p} \cap \mup^{-1}(0) \neq \emptyset \}$.
\begin{theorem}\label{main3}
Assume that $Z$ is $U^\C$-semistable and $M$ is a $G$-semistable real connected submanifold of $Z$. Assume also
$M$ is contained in the zero fiber of $\mu_{i\mathfrak k}$. Then the square of the $G$-gradient map $\parallel \mup \parallel^2$ is constant  if and only if $G$ is semisimple and $M$ is $K$-equivariantly isometric to the product of a real flag and an embedded closed submanifold which is acted on trivially by $K$.
\end{theorem}
\section{Closed orbits and gradient map}\label{se1}
Let $(Z, \omega)$ be a \Keler manifold. Assume that $U^\C$ acts
holomorphically on $Z$, that $U$ preserves $\omega$ and that there is a
momentum map for the $U$ action on $Z$. Let $G\subset U^\C$ be a closed compatible subgroup and let $M$ be a $G$-invariant submanifold of $(Z,\omega)$ and let $\mup:M \lra \liep$ be the associated $G$-gradient map.
\begin{lemma}\label{meo}
Let $x\in M$. Then:
\begin{itemize}
\item if $x$ realizes a  local maximum of $\nu_\liep$, then $G\cdot x=K\cdot x$ and so it is compact;
\item if $G\cdot x$ is compact, then $G\cdot x=K\cdot x$ and $x$ is  a critical point of $\nu_\liep$.
\end{itemize}
\end{lemma}
\begin{proof}
If $x$ realizes a local maximum for $\nu_\liep$, then $\nu_\liep:G \cdot x \lra \R$ has a local maximum at $x$. By Corollary $6.12$, $p.21$ in \cite{heinzner-schwarz-stoetzel}, it follows $G\cdot x=K\cdot x$.

Assume $G\cdot x$ is compact. Then $\nu_p :G\cdot x \lra \R$ has a local maximum. Applying, again, Corollary $6.12$ $p.21$ in \cite{heinzner-schwarz-stoetzel}, we get $G\cdot x=K\cdot x$. We compute the differential of $\nu_\liep$ at $x$. It is easy to check
\[
\mathrm d \nu_\liep (v)=2\langle (\mathrm d \mup)_x (v), \mup(x) \rangle.
\]
Therefore, keeping in mind that $\mathrm{Ker}\, (\mathrm d \mup)_{x}= (\liep \cdot x)^\perp$, where $\liep \cdot x=\{\xi_Z (x):\, \xi \in \liep\}$ see \cite{heinzner-schwarz-Cartan}, it follows $\mathrm ( d \nu_p)_x =0$ on $(\liep \cdot x)^\perp$. Since $G\cdot x=K\cdot x$, it follows $\liep  \cdot x \subset \mathfrak k \cdot x$ and so, keeping in mind that $\nu_\liep$ is $K$-invariant, $(\mathrm d \nu_p)_x=0$ on $ \liep  \cdot x$ as well, proving $x$ is a critical point of $\nu_\liep$.
\end{proof}
\begin{lemma}\label{urka}
Let $x\in M$ be such that $G\cdot x$ is compact. Let $\beta=\mup(x)$. Then
\[
\mathfrak k \cdot x= \mathfrak p \cdot x  \stackrel{\perp}{\oplus} \mathfrak k^{\beta} \cdot x.
\]
Therefore $\mathfrak k\cdot x=\mathfrak p \cdot x$ if and only if $\dim K\cdot x=\dim K \cdot \beta$.
\end{lemma}
\begin{proof}
Since $G\cdot x$ is compact, by the above Lemma $G\cdot x= K\cdot x$. By the $K$-equivariance of $\mup$, it follows that $\mup:K\cdot x \lra K\cdot \beta$ is a smooth fibration. Therefore, keeping in mind that $\mathrm{Ker}\, (\mathrm d \mup)_x=(\mathfrak p \cdot x)^\perp$, we have
\[
(\mathfrak p \cdot x)^\perp \cap \mathfrak k \cdot x=\mathfrak k^{\beta}\cdot x.
\]
Since $G\cdot x=K\cdot x$, we get
\[
\mathfrak k \cdot x= \mathfrak p \cdot x  \stackrel{\perp}{\oplus} ((\mathfrak p \cdot x)^\perp \cap \mathfrak k \cdot x)=\mathfrak p \cdot x  \stackrel{\perp}{\oplus} \mathfrak k^{\beta} \cdot x.
\]
This also implies $\mathfrak k\cdot x=\mathfrak p \cdot x$ if and only if $\dim K\cdot x=\dim K \cdot \beta$, concluding the proof.
\end{proof}
Assume that $G$ is a real form of $U^\C$. If $G\cdot x$ is closed then it is not in general true that $U^\C \cdot x$ is closed. Indeed, let $V$ be a complex vector space and let $\tau:G \lra \mathrm{PGL}(V)$ be an irreducible faithful projective representation. Since the center of $G$ acts trivially, we may assume that $G$ is semisimple. The representation $\tau$ extends to an irreducible projective representation of $U^\C$. It is well-known that $U^\C$ has a unique closed orbit \cite{guillemin}. It is the orbit throughout a maximal vector. On the other hand $G$ could have more than one closed orbit in $\mathbb P(V)$ \cite[Proposition 4.28, p. $58$]{gjt}. The following result tells us that there exists a unique closed $G$-orbit contained in the unique closed orbit of $U^\C$.
\begin{proposition}
Let $M=U^\C \cdot x$ be a compact orbit. If $G$ is a real form of $U$, then there exists exactly one closed $G$-orbit in $M$.
\end{proposition}
\begin{proof}
$U^\C \cdot x=U\cdot x$ and it is a flag manifold \cite{heinzner-schwarz-stoetzel,guillemin}.  Applying a beautiful old Theorem of Wolf \cite{wolf}, it follows that $G$ has a unique closed orbit in $M$. The $G$ orbit is given by the orbit throughout the maximum of the norm square of the gradient map \cite{heinzner-schwarz-stoetzel}.
\end{proof}
The following result arises from Lemma \ref{urka}.
\begin{cor}
Let $x\in M$ be such that $G\cdot x$ is compact. If $\dim K\cdot x=\dim K \cdot \mup(x)$, then $U^\C\cdot x$ is closed.
\end{cor}
\begin{proof}
Since $\mathfrak u=\mathfrak k \oplus i\mathfrak p$, it follows $\mathfrak u \cdot x= \mathfrak k \cdot x+ i\mathfrak p \cdot x$. By Lemma \ref{urka}, $\mathfrak k\cdot x=\mathfrak p \cdot x$ and so $\mathfrak u^\C \cdot x=\mathfrak u \cdot x$. This implies $U\cdot x$ is open and closed in  $U^\C\cdot x$. Therefore $U^\C \cdot x=U\cdot x$, concluding the proof.
\end{proof}
The following result gives a necessary and sufficient condition such that $U^\C \cdot x$ is closed whenever $G\cdot x$ is.
\begin{proposition}
Let $x\in M$ be such that $G\cdot x$ is compact. If $G$ is a real form of $U^\C$, then $U^\C\cdot x$ is closed if and only if $i \mathfrak k^{\mup(x)}\cdot x \subseteq \mathfrak u \cdot x \cap i(\mathfrak p\cdot x)^\perp$. If $M$ is Lagrangian, then $U^\C \cdot x$ is closed if and only if $\mup:K\cdot x \lra K\cdot \mup (x)$ is a covering map. Moreover, $G\cdot x$ is a Lagrangian submanifold of $U^\C\cdot x$.
\end{proposition}
\begin{proof}
Set $\beta=\mup(x)$. By Lemma \ref{urka},  $\mathfrak k \cdot x= \mathfrak p \cdot x  \stackrel{\perp}{\oplus} \mathfrak k^{\beta} \cdot x$. Therefore, keeping in mind $\mathfrak u =\mathfrak k \oplus i\mathfrak p$, we have
\[
\mathfrak u \cdot x=\mathfrak p \cdot x  \stackrel{\perp}{\oplus} \mathfrak k^{\beta} \cdot x + i\mathfrak p\cdot x.
\]
Since $i\mathfrak k^{\beta} \cdot x$ is orthogonal to $i\mathfrak p \cdot x$, it follows that
$
\mathfrak u \cdot x=\mathfrak u^\C \cdot x,
$
if and only if $i\mathfrak k^{\beta} \cdot x \subset \mathfrak u \cdot x \cap i(\mathfrak p \cdot x)^\perp$. If $M$ is Lagrangian, then $T_x Z=T_x M \stackrel{\perp}{\oplus} J (T_x M)$. Therefore
\[
\mathfrak u \cdot x=\mathfrak p \cdot x  \stackrel{\perp}{\oplus} \mathfrak k^{\beta} \cdot x \stackrel{\perp}{\oplus}  i\mathfrak p\cdot x.
\]
This implies $\mathfrak u \cdot x=\mathfrak u^\C \cdot x$ if and only if $i\mathfrak k^{\beta} \cdot x \subseteq i \mathfrak p \cdot x$. By the first part of the proof we get $U^\C \cdot x$ is compact if and only if $\mathfrak k^{\beta} \cdot x=\{0\}$ and so if and only if $\dim K\cdot x =\dim K\cdot \beta$. In particular $\mathfrak p\cdot x =\mathfrak k\cdot x$. This implies $
\dim_{\R} G\cdot x = \dim_{\C} U^\C \cdot x$ and so $G\cdot x$ is a compact Lagrangian submanifold of $U^\C \cdot x$.
\end{proof}
\begin{proposition}
Let $M$ be a $G$-invariant Lagrangian submanifold of $(Z,\omega)$.  Let $x\in M$. Then $U^\C\cdot x$ is closed if and only if $\mathfrak k \cdot x=\mathfrak p \cdot x$. In particular $G\cdot x$ is closed and it is a Lagrangian submanifold of $U^\C \cdot x$.
\end{proposition}
\begin{proof}
Since $M$ is Lagrangian, we have
\[
\mathfrak u \cdot x =\mathfrak k \cdot x \stackrel{\perp}{\oplus} i \mathfrak p \cdot x.
\]
Therefore $\mathfrak u \cdot x =\mathfrak u^{\C} \cdot x$ if and only if $i \mathfrak k \cdot x \subseteq i \mathfrak p \cdot x$ and $\mathfrak p \cdot x \subseteq  \mathfrak k \cdot x$ hence if and only if $\mathfrak k \cdot x=\mathfrak p \cdot x$. This also implies $G\cdot x$ is compact, $
\dim_{\R} G\cdot x = \dim_{\C} U^\C \cdot x$ and so $G\cdot x$ is a compact Lagrangian submanifold of $U^\C \cdot x$.
\end{proof}
\begin{proposition}
Let $x\in Z$. Assume that both $G\cdot x$ and $U^\C \cdot x$ are compact. Then $\dim_{\R} \U^\C \cdot x \leq 2 \dim G\cdot x$. If the equality holds then $G\cdot x$ is totally real.
\end{proposition}
\begin{proof}
By Lemma \ref{meo} $U^\C \cdot x=U\cdot x$ and $G\cdot x=K\cdot x$.
Since $\mathfrak u \cdot x =\mathfrak k \cdot x + i \mathfrak p \cdot x$ and $\mathfrak p \cdot x \subseteq \mathfrak k \cdot x$, it follows that
\[
\dim_{\R} U^\C \cdot x \leq 2 \dim G\cdot x.
\]
Note also that $\mathfrak k^\C \cdot x =\mathfrak u^\C \cdot x$. This implies $K^\C \cdot x$ is open in $U^\C \cdot x$. This remark is not new, see \cite{heinzner-schwarz-stoetzel,heinzner-stoetzel-global}, and it arises from the Matsuki duality \cite{matsuki}.
 Finally, $2\dim G\cdot x=\dim_{\R} U^\C$ if and only if $\mathfrak k\cdot x=\mathfrak p \cdot x$ and $\mathfrak u \cdot x =\mathfrak k \cdot x \oplus  i \mathfrak p \cdot x$. In particular $G\cdot x$ is totally real in $U^\C \cdot x$.
\end{proof}
The momentum map of $U$ on $Z$ induces a gradient map $\mu_{i\mathfrak k}$ of $K^\C$ in $Z$. Assume that
$M$ is contained in the zero fiber of $\mu_{i\mathfrak k}$. 
\begin{lemma}
Let $x\in M$. If $U^\C \cdot x$ is closed, then $G\cdot x$ is closed.
\end{lemma}
\begin{proof}
Let $y\in U^\C \cdot x$. Since $\mu=\mu_{i\mathfrak k} + \mup$,  it follows that
\[
\parallel \mup (y) \parallel^2 \leq \parallel \mu (y) \parallel^2=\parallel \mu(x) \parallel^2=\parallel \mup(x) \parallel^2.
\]
Hence $\nu_\liep: U^\C \cdot x \lra \R$ achieves its maximum in $x$. By Lemma \ref{meo}, $G\cdot x$ is closed.
\end{proof}
We say that $M$ is $G$-semistable if $M=\{p\in M:\, \overline{U^\C \cdot p} \cap \mup^{-1}(0)\}$. In the papers \cite{heinzner-schwarz-Cartan,heinzner-schwarz-stoetzel}, the authors proved if $M$ is $G$-semistable then $G\cdot x$ is closed if and only if $G\cdot x \cap \mup^{-1}(0) \neq \emptyset$. As an application we get the following result.
\begin{proposition}\label{closed-orbit}
Assume that  $(Z,\omega)$ is $U^{\C}$-semistable and $M$ is $G$-semistable and it is contained in the zero fiber of $\mu_{i\mathfrak k}$.
Let $x\in M$. Then $G\cdot x$ is closed if and only if $U^\C \cdot x$ is closed.
\end{proposition}
\begin{proof}
By the above result it is enough to prove  if $G\cdot x$ is closed then $U^\C \cdot x$ is closed. If $G\cdot x$ is closed then $G\cdot x \cap \mup^{-1}(0) \neq \emptyset$. Since $\mup^{-1}(0) \cap M=\mu^{-1}(0) \cap M$, the result follows.
\end{proof}
A corollary we prove a well-known result of Birkes \cite{birkes}, see also \cite{borel-chandra}
\begin{cor}
Let $G$ be a real form of $U$. Let $V$ be complex vector space and $W$ be real subspace of $V$ such that $V=W^\C$. Assume that $G$ acts on $W$. Let $w\in W$. Then $G\cdot w$ is closed if and only if $U^\C \cdot x$ is closed.
\end{cor}
\begin{proof}
It is well-known that $V$, respectively $W$, is $U^\C$-semistable, respectively $G$-semistable \cite{rs}, see also \cite{biliosp}. Since $W$ is a Lagrangian subspace of $V$, applying the above Proposition the result follows.
\end{proof}
\section{norm square of the gradient map}\label{se2}
We investigate splitting results for $G$-invariant real submanifolds of $(Z,\omega)$.
\begin{proposition}\label{meo2}
Let $M$ be a $G$-stable connected submanifold of $Z$ and let $\mup:M \lra \liep$ be the restricted gradient map. Then the square of the gradient map $\nu_\liep:M \lra \R$ is constant  if and only if any $G$ orbit is compact.
\end{proposition}
\begin{proof}
Assume $\nu_p$ is constant. Let $x\in M$. Then $\nu_\liep:G\cdot x \lra \R$ is constant and so $\nu_\liep$ has a maximum on $x$. By Lemma \ref{meo} $G\cdot x=K\cdot x$ and so it is compact. Vice-versa, assume that any $G$ orbit is compact. By Lemma \ref{meo} $(\mathrm d \nu_p)_x=0$   for any $x\in M$. Since $M$ is connected it follows $\nu_\liep$ is constant.
\end{proof}
The following result is proved in \cite{heinzner-schwarz-stoetzel}. For the sake of completeness we give a proof.
\begin{proposition}\label{splitting}
Let $M$ be a $G$-stable connected submanifold of $Z$ and let $\mup:M \lra \liep$ be the restricted gradient map.
If  $\nu_\liep$ is constant, then $\mup (M)=K\cdot \beta$, $\mu^{-1}(\beta)$ is a submanifold and the following splitting
\[
M=K\times_{K^\beta} \mup^{-1}(\beta),
\]
holds.
\end{proposition}
\begin{proof}
Since $\nu_\liep$ is constant,  it follows that $M=S_\beta$, where $S_\beta$ is the maximal strata, and  $\mup(S_\beta)=\mup(M)=K\cdot \beta$ \cite[$p.21$]{heinzner-schwarz-stoetzel}. In particular $M=K\mup^{-1}(\beta)$ and we may think $\mup:M \lra K\cdot \beta$. Therefore  $\beta$ is a regular value and so $\mup^{-1}(\beta)$ is a $K^\beta$-invariant submanifold of $M$.

Let $x\in \mup^{-1}(\beta)$. By the $K$-equivariance of $\mup$, it is easy to check $K\cdot x \cap \mup^{-1} (\beta)= K^\beta \cdot x$. We claim that the same holds infinitesimally, i.e., $T_x \mup^{-1} (\beta) \cap \liek \cdot x=\liek^\beta \cdot x$. Indeed, let $v \in T_x \mup^{-1} (\beta) \cap \liek \cdot x$. Let $\xi \in \liek$ such that $v=\xi_M (x)$. Since $T_x \mup^{-1}(\beta)=\mathrm{Ker}\, (\mathrm{d}\, \mup)_x$,
 we get
\[
0=\desudtzero \mup(\exp(t\xi)x)=\desudtzero \mathrm{Ad}(\exp(t\xi)) \beta,
\]
and so $v\in \mathfrak k^\beta \cdot x$.

We define the map
\[
\Psi:K \times_{K^\beta} \mup^{-1}(\beta) \lra M \qquad [k,x] \mapsto kx.
\]
It is easy to check that $\Psi$ is $K$-equivariant and smooth.
Since $\mup(M)=K\cdot \beta$ it follows  $M=K\cdot \mup^{-1}(\beta)$ and so $\Psi$ is surjective. It is also injective since $kx=k'x$ if and only if $k'^{-1}k\in K^\beta$, proving it is bijective. Now, we proof that $\Psi$ is a local diffeomorphism. This implies that $\Psi$ is a diffeomorphism concluding the proof. Note that it is enough to prove $\mathrm d \Psi_{[e,x]}$ is a diffeomorphism by the $K$-equivariance. Now,
\[
T_x M=(\mathfrak p \cdot x)\oplus (\mathfrak p \cdot x)^\perp =(\mathfrak p \cdot x) \oplus T_x \mup^{-1}(\beta).
\]
By Proposition \ref{meo2} any $G$ orbit is a $K$ orbit. This implies $\mathfrak p \cdot x\subset \mathfrak k \cdot x$. Since $\liek^\beta \cdot x \subset (\liep \cdot x)^\perp$, it follows that the map
\[
\liep \cdot x \hookrightarrow \liek \cdot x \longrightarrow \liek \cdot x / \liek^\beta \cdot x,
\]
is injective. Therefore $\mathrm d \Psi_{[e,x]}$ is surjective. Since $\Psi$ is bijective it follows that $\mathrm d \Psi_{[e,x]}$ must be bijective.
\end{proof}
We are ready to prove the splitting results.
\begin{proof}[Proof of Theorem \ref{main2}]
Since $\nu$ is constant, applying Lemma \ref{meo2} it follows that any $U^\C$ orbit is compact and it is a complex $U$ orbit. Then for any $x\in M$, we have  $U_{x}=U_{\mu(x)}$ \cite{guillemin}. Since $U_{\mu(x)}$ is a centralizer of a torus, then the center of $U$ does not act on $M$ and so $U$ is semisimple. By the above proposition $M=U/U^\beta \times \mu^{-1}(\beta)$ and for very $x\in \mu^{-1}(\beta)$, $U_x=U^\beta$ and so $U_x$ acts trivially on $\mu^{-1}(\beta)$. If $x\in \mu^{-1}(\beta)$, then
\[
T_x M=(i\liu \cdot x)\stackrel{\perp}{\oplus} T_x \mu^{-1}(\beta)=T_x U\cdot x \stackrel{\perp}{\oplus} T_x \mu^{-1}(\beta).
\]
This implies that the $U$ action on $M$ is polar with section $\mu^{-1}(\beta)$ \cite{dadok} and so $\mu^{-1}(\beta)$ is totally geodesic. We claim that the above splitting is Riemannian.

Let $\xi \in \liu$ and let $\xi_M$ the induced vector field. It is enough to prove that the function $g(\xi_M, \xi_M)$ is constant when restricted to $\mu^{-1} (\beta)$.

Let $x\in \mu^{-1}(p)$ and $v\in T_x \mu^{-1}(p)$. We may extend $v$ to a vector field on a neighborhood of $p$, that we denote by $X$,  such that $g(X,\xi_M)=0$ for any $z \in W$ and for any $\xi\in \liu$.  Indeed, let $\xi_1,\ldots,\xi_k \in \liu$ such that $(\xi_1)_M (x), \ldots (\xi_k)_M (x)$ is a basis of $T_x U\cdot x$. Since the $U$ action on $M$ has only one type of orbit, it follows that there exists a neighborhood $W$ of $x$ such that $(\xi_1)_M (y), \ldots, (\xi_k)_M (y)$ is a basis of $T_y U\cdot y$ for any $y\in W$. Applying a Gram-Schmidt process we get an orthonormal basis $\{Y_1,\ldots,Y_k\}$ of $T_y U\cdot y$ for any $y\in W$. Let $\tilde X$ any local extension of $v$. Then
\[
X=\tilde X -g(Y_1,\tilde X) Y_1 -\cdots -g(Y_k,\tilde X) Y_k,
\]
satisfies the above conditions. Moreover, for any $z\in \mup^{-1}(\beta) \cap W$, the vector field $X$ lies in $T_z \mup^{-1}(\beta)$ due to the orthogonal splitting $T_z M=T_z U\cdot z \stackrel{\perp}{\oplus} T_z \mup^{-1}(\beta)$.

Let $\nu_M=-J(\xi_M)$ Then  $J(\nu_M)=\xi_M$. Since $M=U/U_\beta \times \mu^{-1}(p)$, it follows $[X,\xi_M]=[X,\nu_M]=0$ along $\mup^{-1}(\beta)$. By the closeness of $\omega$, we have
\[
\mathrm d \omega (v,\nu_M(x),\xi_M (x))=0.
\]
On the other hand, by the Cartan  formula \cite{kobayashi}, we have
\[
\begin{split}
\mathrm d \omega (v,\nu_M(x),\xi_M (x))&=X\omega(\nu_M, \xi_M) + \nu_M \omega (\xi_M,X)+ \xi_M \omega (X,\nu_M ) \\
&-\omega([X,\nu_M],\xi_M)-\omega([\nu_M,\xi_N],X)-\omega([\xi_M, X],Y).
\end{split}
\]
Now, $\omega([X,\nu_M],\xi_M)=\omega([\xi_M, X],Y)=0$ due to the fact that $[X,\nu_M](x)=[\xi_M, X](x)=0$, The term $\omega([\nu_M,\xi_N],X)=0$, since
\[
\omega([\nu_M,\xi_N],X)=g(J([\nu_M,\xi_M],X)=0
\]
 due to the facts that the $U$ orbit is complex and the splitting $T_x M=T_x \mu^{-1}(\beta) \stackrel{\perp}{\oplus} T_x U\cdot x$ holds. Finally, $
\nu_M \omega (\xi_M,X)=0$, respectively $\xi_M \omega (X,\nu_M )=0$, due to the fact that
\[
\omega(\xi_M,X)=g(J \xi_M, X)=0,
\]
respectively,
\[
\omega(X,\nu_M)=g(JX,\nu_M)=-g(X,J\nu_M)=0,
\]
along $U\cdot x$. Therefore
\[
0=\mathrm d \omega (v,\nu_M(x),\xi_M (x))=X\omega(\nu_M,\xi_M)=Xg(J (\nu_M),\xi_M)=Xg(\xi_M,\xi_M),
\]
and so  $g(\xi_M,\xi_M)$ is constant along $\mup^{-1}(\beta)$ and the result is proved.
\end{proof}
\begin{proof}[Proof of Theorem \ref{main3}]
By Proposition \ref{splitting} $M=K \times_{K^\beta} \mup^{-1}(\beta)$. By Proposition \ref{meo2} it follows $U^\C \cdot x$ is compact for any $x\in \mup^{-1}(\beta)$. Let $x\in \mup^{-1}(\beta)$.
By Proposition \ref{closed-orbit}, $U^\C \cdot x$ is compact as well and $\mup(x)=\mu(x)=\beta$. This implies $K_x=K\cap U_x=K\cap U^{\beta}=K^{\beta}$ for any $x\in \mup^{-1}(\beta)$ and so $M=K/K^{\beta} \times \mup^{-1}(\beta)$. The Lie algebra of the center of $G$ is contained in the Lie algebra of the center of $U^\C$. On the other hand, the Lie algebra of the center of $U^\C$ is the complexification of the Lie algebra of the center of $U$ which acts trivially on $M$. This implies $G$ is semisimple. Finally, keeping in mind that
$\omega$ is closed and $U^\C \cdot x$ is compact for any $x\in \mup^{-1}(\beta)$, applying the same idea of the above proof we get the splitting $M=K/K^{\beta} \times \mup^{-1}(\beta)$ is Riemannian.
\end{proof}


\begin{thebibliography}{99}
\bibitem{bg}
\textsc{L. Bedulli and A. Gori}
\newblock A splitting result for compact symplectic manifolds,
\newblock \emph{Results Math.} \textbf{47} (2005), 194--198.
%
\bibitem{birkes}
\textsc{D. Birkes},
\newblock Orbits of linear algebraic groups,
\newblock \emph{Ann. Math.} \textbf{93} (2) (1971) 459--475.
%
\bibitem{bs}
\textsc{L. Biliotti},
\newblock A note on moment map on symplectic manifolds,
\newblock \emph{Bull. Belg. Math. Soc. Simon Stevin} \textbf{16} (2009), 107--116.
%
\bibitem{biliosp}
\textsc{L. Biliotti},
\newblock The Kempf-Ness Theorem and invariant theory for real reductive representations,
\newblock \emph{SPJM} published on-line 23 September 2019.
%
%
%
%
%
\bibitem{borel-chandra}
\textsc{A. Borel and Harish-Chandra},
\newblock Arithmetic subgroups of algebraic groups,
\newblock \emph{Ann. Math.} \textbf{75} (2) (1962) 485--535.
%
\bibitem{dadok}
\textsc{J. Dadok},
\newblock Polar coordinates induced by actions of compact Lie groups,
\newblock \emph{Trans. Amer. Math. Soc.} \textbf{288} (1) (1985), 125--137.
%
\bibitem{pg}
\textsc{A. Gori and F. Podest\`a},
\newblock A note on the moment map on compact K\"ahler manifold,
\newblock \emph{Ann.Global. Anal. Geom.} \textbf{26} (2004), 315--318.
%
\bibitem{guillemin}
\textsc{V. Guillemin and S. Sternberg},
\newblock \emph{Symplectic techniques in physics, 2nd ediction},
\newblock Cambridge University Press, Cambridge, 1990.
%
\bibitem{gjt}
\textsc{Y. Guivarc'h, L. Ji, and J. C. Taylor},
\newblock \emph{Compactifications of symmetric spaces}.
\newblock Progress in Mathematics \textbf{156} Birkh\"auser Boston Inc., Boston, MA, 1998
%
%
%
\bibitem{heinzner-schwarz-Cartan}
\textsc{P.~Heinzner and  G.~W. Schwarz},
\newblock Cartan decomposition of the moment map,
\newblock \emph{Math. Ann.} \textbf{337} (1) (2007) 197--232.
%
\bibitem{heinzner-schwarz-stoetzel}
\textsc{P.~Heinzner, G.~W. Schwarz and H.~St{\"o}tzel},
\newblock Stratifications with respect to actions of real reductive groups,
\newblock \emph{Compos. Math} \textbf{144} (1) (2008) 163--185.
%
\bibitem{heinzner-stoetzel-global}
\textsc{P.~Heinzner and H.~St{\"o}tzel},
\newblock Critical points of the square of the momentum map,
\newblock \emph{Global aspects of complex geometry},
211--226.  Springer, Berlin, 2006.
%
%
%
\bibitem{helgason}
\textsc{S. Helgason},
\newblock \emph{Differential Geometry, Lie Groups and Symmetric Spaces},
\newblock Corrected reprint of the 1978 original. Graduate Studies in Mathematics, 34. American Mathematical Society, Providence, RI, 2001.
%
\bibitem{kirwan}
\textsc{F.~C. Kirwan},
\newblock \emph{Cohomology of quotients in symplectic and algebraic geometry},
\newblock volume \textbf{31} of Mathematical Notes, Princeton University Press,
  Princeton, NJ, 1984.
%
\bibitem{knapp-beyond}
\textsc{A.~W. Knapp},
\newblock \emph{Lie groups beyond an introduction},
\newblock volume \textbf{140}, Progress in Mathematics, Birkh\"auser Boston Inc., Boston, MA, second edition, 2002.
%
%
\bibitem{kobayashi}
\textsc{S. Kobayashi and K.  Nomizu},
Foundations of differential geometry. Vol. I. John Wiley \&
Sons, Inc., New York (1996)
%
%
\bibitem{rs}
\textsc{R. W. Richardson and P.J. Slodowoy},
\newblock Minumun vectors for real reductive algebraic groups,
\newblock \emph{J. London Math. Soc.} \textbf{42} (2) (1990) 409--429.
%
%
%
%
%
\bibitem{matsuki}
\textsc{T. Matsuki},
\newblock Orbits on affine symmetric spaces under the action of parabolic subgroups
\newblock \emph{Hiroshima Math. J.} \textbf{12} (1982), 307--320.
%
\bibitem{wolf}
\textsc{J. Wolf},
\newblock The action of a real semisimple group on a complex flag manifold I. Orbit structure and holomorphic arc,
\newblock \emph{Bull. Amer. Math. Soc.} \textbf{75} (1969) 1121--1237.
\end{thebibliography}
\end{document}